\documentclass[reqno,10pt]{amsart}
\usepackage{xspace}
\usepackage[bookmarksnumbered,colorlinks]{hyperref}
\usepackage{graphics}
\newcommand{\bt}{\begin{theorem}}                               
\newcommand{\et}{\end{theorem}}                                 
\newcommand{\bd}{\begin{definition}}                            
\newcommand{\ed}{\end{definition}}                              
\newcommand{\bl}{\begin{lemma}}                                 
\newcommand{\el}{\end{lemma}}                                   
\newcommand{\bpr}{\begin{proposition}}                  
\newcommand{\epr}{\end{proposition}}                    
\newcommand{\bere}{\begin{remark}}                      
\newcommand{\ere}{\end{remark}}                                 
\newcommand{\beq}{\begin{equation}}
\newcommand{\eeq}{\end{equation}}
\def\bal#1\eal{\begin{align}#1\end{align}}                      
\def\baln#1\ealn{\begin{align*}#1\end{align*}}          
\def\bml#1\eml{\begin{multline}#1\end{multline}}        
\def\bmln#1\emln{\begin{multline*}#1\end{multline*}}  
\def\bga#1\ega{\begin{gather}#1\end{gather}}
\def\bgan#1\egan{\begin{gather*}#1\end{gather*}}
\newcommand{\de}{\mathrm{d}}                        
                     
\newcommand{\N}{\ensuremath{\mathbb{N}}\xspace}     
\newcommand{\R}{\ensuremath{\mathbb{R}}\xspace}     
\newcommand{\eps}{\varepsilon}

\newtheorem{theorem}{Theorem}[section]

\newtheorem{lemma}[theorem]{Lemma}
\newtheorem{proposition}[theorem]{Proposition}
\theoremstyle{definition}
\newtheorem{definition}[theorem]{Definition}
\theoremstyle{remark}
\newtheorem{remark}[theorem]{Remark}
\newcommand{\dist}{\ensuremath{\mathrm{d}}\xspace}
\newcommand{\distg}{\ensuremath{\mathrm{dist}^h}\xspace}

\newcommand{\nablag}{\ensuremath{\nabla^h}\xspace}

\author[E. Caponio]{Erasmo Caponio}
\address{Dipartimento di Matematica, Politecnico di Bari, Via Orabona 4,
70125, Bari, Italy}
\email{caponio@poliba.it}

\author[M. A. Javaloyes]{Miguel Angel Javaloyes}
\address{Departamento de Geometr\'{\i}a y Topolog\'{\i}a.
 Facultad de Ciencias, Universidad de Granada.
 Campus Fuentenueva s/n, 18071 Granada, Spain}
\email{ma.javaloyes@gmail.com, majava@ugr.es}

\author[A. Masiello]{Antonio Masiello}
\address{Dipartimento di Matematica,
Politecnico di Bari, Via Orabona 4,
70125, Bari, Italy}
\email{masiello@poliba.it}

\thanks{EC and AM are supported by M.I.U.R. Research project PRIN07 ``Metodi Variazionali  e Topologici nello Studio di Fenomeni Nonlineari''}

\thanks{MAJ is partially supported by Regional J.
Andaluc\'{\i}a Grant P06-FQM-01951, by Fundaci\'on S\'eneca project 04540/GERM/06   and by Spanish MEC Grant MTM2009-10418.}

\title[Finsler geodesics in the presence of a convex function]%
{Finsler geodesics in the presence of a convex function and their applications}



\date{}

\begin{document}
\begin{abstract}
In this paper, we obtain a result about the existence of only a finite number of
geodesics between two fixed non-conjugate points in a Finsler
manifold  endowed with a  convex function. We apply it  to Randers
and Zermelo metrics. As a by-product, we also get a result about the
finiteness of the number of lightlike and timelike geodesics
connecting an event to a line in a standard stationary spacetime.
\end{abstract}

\maketitle
\begin{section}{Introduction}
 In this paper, we extend to Finsler metrics a result about the
finiteness of the number of geodesics joining two fixed points on a
Riemannian manifold, see \cite{GiMaPi99}. Moreover, we present two
applications of this abstract result. First we show that, under
suitable assumptions, the number of lightlike or timelike geodesics
with fixed arrival proper time joining an event and a timelike
curve in a stationary spacetime is finite.  Afterwards,  we study the
finiteness of geodesics joining two given points in a manifold
endowed with a Zermelo metric.

Let $(M,F)$ be a  non-reversible Finsler manifold; then two conditions
of completeness are available: the forward and the
backward completeness. As a consequence of the non-reversibility of the
metric, the distance naturally associated to  a   Finsler metric is
not symmetric. The distance $\dist(p,q)$ between two points $p$ and
$q$ of $M$ is defined as the infimum of all the lengths, with
respect to the Finsler structure $F$, of curves joining $p$ and $q$
on $M$, so it is
\[
 \dist (p,q)=\inf_{\gamma\in  \Omega(p,q)  }\int_0^1 F(\gamma,\dot\gamma)\dist s,
\]
where $\Omega(p,q)$ is the set of all the piecewise smooth curves
from $p$ to $q$. A {\it forward (backward)} Cauchy sequence is a
sequence $\{x_i\}_{i\in\N}$ such that for every $\varepsilon>0$
there exists $N\in\N$ with $\dist(x_i,x_j)<\varepsilon$ for every
$j>i\geq N$ ($i>j\geq N$). The Finsler manifold $(M,F)$ is said {\em
forward (backward) complete} if all the forward (backward) Cauchy
sequences converge. By the Finslerian Hopf-Rinow theorem (see
\cite[Theorem 6.6.1 and Exercise 6.6.7]{BaChSh00}) forward
(backward) completeness is equivalent to forward (backward)
geodesic completeness. We recall that the  curve with the reverse
parametrization  of a geodesic for a non-reversible Finsler
metric is not necessarily a geodesic. For this reason, we say that
the metric is forward (backward) geodesically complete when
geodesics with constant speed can be extended up to $+\infty$ (up to
$-\infty$).

The  main  result of the paper is the following: given a  forward or
backward complete   Finsler manifold that admits a $C^2$ strictly
convex function having  a  non-degenerate  minimum point,  then the
number of geodesics between two non-conjugate points is finite (see
Theorem \ref{geonfinito}).  We will also study the existence of such
convex functions  for Randers, Zermelo and Fermat metrics.

Randers metrics were introduced in \cite{Rander41} in order to study
electromagnetic fields in general relativity. Zermelo metrics were
introduced in \cite{Zermel31} to study the least time travel path of
a body moving under the influence of a mild wind. Fermat metrics are
a particular type of Randers metrics defined on a spacelike
hypersurface of a standard stationary spacetime. They come into play
in the development of a variational theory for lightlike or timelike
geodesics on a standard stationary spacetime,  see \cite{CJM08}.
Such variational theory allows one to give a mathematical model for
the {\it gravitational lensing effect} in astrophysics, see
\cite{GMPJMP,PE}.

Randers, Fermat and Zermelo metrics provide the same family of
Finsler metrics (see  for example \cite[Proposition 3.1]{BiJa08}),
but they are defined adding to a Riemannian metric on a manifold $M$
a different geometric object, as a vector field, a positive function
or a one-form. For this reason we study them separately. More
precisely, since several results are known on the existence of
convex functions in Riemannian Geometry (see Section
\ref{randerssection}), we shall study when a convex function for a
Riemannian metric on a manifold $M$ still remains convex passing to
one of the Finsler structures above (see
Propositions~\ref{convessoRanders}, \ref{improve} and
\ref{Zermelofinite}).

The paper is structured as follows. In Section \ref{finito} we give
some  basic notions about Finsler geometry and we obtain the main
result about the existence of a finite number of geodesics joining
two fixed points in the presence of a convex function (see Theorem
\ref{geonfinito}). Section \ref{appli} is devoted to applications.
In subsection \ref{randerssection} we obtain a finiteness result for
Randers metrics (see Proposition \ref{finitezzaRanders}). In
subsection \ref{stationaryspacetimes} we use the Fermat metric to
obtain some results about the finiteness  of the number of lightlike
geodesics or timelike geodesics with fixed  arrival proper time,
between an event and a stationary observer (see Proposition
\ref{improve}). Finally in subsection \ref{Zermelo} we deduce a
finiteness result for Zermelo metrics (see Proposition
\ref{Zermelofinite}).

\end{section}

\begin{section}{A finiteness result in the presence of a convex function}\label{finito}
Let $M$ be a smooth, connected, finite dimensional manifold and let
$TM$ be the tangent bundle of $M$; a non-reversible Finsler metric on
$M$ is a function $F\colon TM\to[0,+\infty)$ which is
\begin{enumerate}
\item[1)] continuous on $TM$, $C^{\infty}$ on $TM\setminus 0$,
\item[2)] fiberwise positively homogeneous of degree one, i. e.   $F(x,\lambda y)=\lambda F(x,y)$, for all $x\in M$,   $y\in T_x M$ and  $\lambda>0$,
\item[3)] the square $F^2$ is fiberwise strictly convex i.e. the matrix
\[
g_{ij}(x,y)=\left[\frac{1}{2}\frac{\partial^2 (F^2)}{\partial y^i\partial y^j}(x,y)\right]
\]
is  positive definite for any $(x,y)\in TM\setminus 0$.
\end{enumerate}
The tensor
\[
g=g_{ij}\de x^i\otimes\de x^j
\]
is called the {\em fundamental tensor} of the Finsler manifold
$(M,F)$; it is a symmetric section of the tensor bundle
$\pi^*(T^*M)\otimes \pi^*(T^*M)$, where $\pi^*(T^*M)$ is the dual of
the pulled-back tangent bundle $\pi^*TM$ over $TM\setminus 0$ ($\pi$
is the projection $TM\to M$).

The {\em Chern connection} $\nabla$ is the unique linear connection
on $\pi^* TM$ whose connection $1$-forms $\omega^{\,\, i}_j$ are
torsion free and almost $g$-compatible (see \cite[Theorem
2.4.1]{BaChSh00}). By using the Chern connection, one can define two
different covariant derivatives $D_T W$ of a smooth vector field $W$
along a smooth regular curve $\gamma=\gamma(s)$ on $M$, with
velocity field
 $T=\dot\gamma$: \baln
&D_{T}W=\left.\left(\frac{\de W^i}{\de t}+W^jT^k\Gamma^i_{\,\, jk}(\gamma,T)\right)\frac{\partial}{\partial x^i}\right\vert_{\gamma(t)}&&\text{\em with reference vector $T$,}\\
&D_{T}W=\left.\left(\frac{\de W^i}{\de t}+W^jT^k\Gamma^i_{\,\,
jk}(\gamma,W)\right)\frac{\partial}{\partial
x^i}\right\vert_{\gamma(t)}&&\text{\em with reference vector
$W$,}\nonumber \ealn where the functions $\Gamma^i_{\,\, jk}$ are
called the components of the Chern connection $\nabla$ and they are
defined by the relation $\omega_j^{\,\, i}=\Gamma^i_{\,\, jk}\de
x^k$ . A geodesic of the Finsler manifold $(M,F)$ is a smooth
regular curve $\gamma$ satisfying the equation
\[
D_{T}\left(\frac{T}{F(\gamma,T)}\right)=0,
\]
with reference vector $T=\dot\gamma$. A curve $\gamma=\gamma(s)$ is
said to have {\em constant speed} if $F(\gamma(s),\dot \gamma (s))$
is constant along $\gamma$. Geodesics are characterized as the
critical points of the length and the energy functionals when
considered in a suitable class of curves joining two points.

Let $(M,F)$ be a Finsler manifold. In analogy with the Riemannian
case we say that a  function $f\colon M\to \R$ is {\em convex}
(resp. {\em strictly convex}) if for every constant speed geodesic
$\gamma\colon I\subset\R\to M$, $f\circ\gamma\colon I\to \R$ is
convex (resp. strictly convex). Let $f\colon M\to\R $ be a  $C^2$
function; a {\em critical point} $x\in M$ of $f$ is a point where
the differential of the function $\de f(x)$ is equal to 0 . The
Finslerian Hessian   $H_f$ of $f$ is the symmetric section of the
tensor bundle $\pi^*(T^*M)\otimes \pi^*(T^*M)$ over $TM\setminus 0$
given by $\nabla(\de f)$, where $\nabla$ is the Chern connection
associated to the Finsler metric $F$. In natural coordinates on
$TM\setminus 0$, the Finslerian Hessian   $H_f$ of $f$ is given by
\[(H_f)_{ij}(x,y)u^iv^j=\frac{\partial^2f}{\partial x^i\partial x^j}(x)u^iv^j-\frac{\partial f}{\partial x^k}(x)\Gamma^k_{\,\,ij}(x,y)u^iv^j.\]
Clearly the functions $(H_f)_{ij}$ are symmetric with respect to the
indexes $i$,$j$, since the components $\Gamma^i_{\,\, jk}$ of the
connection are symmetric with respect to $i$,$j$.

If $\gamma$ is a constant speed geodesic of $(M,F)$, then the second
derivative of the function $g(s)=f(\gamma(s))$ is given by
$g''(s)=(H_f)_{(\gamma(s),\dot\gamma(s))}(\dot
\gamma(s),\dot\gamma(s))$. Thus a $C^2$  function $f$  is convex iff
for every $(x,y)\in TM\setminus 0$, $(H_f)_{(x,y)}(y,y)\geq 0$ and
it is strictly convex if $(H_f)_{(x,y)}(y,y)> 0$  (see also
\cite[Appendix 4]{Udriste}).

A critical point $x$ of $f$ is called {\em non-degenerate} if
$(H_f)_{(x,y)}(y,y)\neq 0$ for any $y\in T_xM$, $y\neq 0$.

The following two propositions are useful to prove the main theorem
of this section. \bpr \label{prop:minimum} Let $(M,F)$ be a  forward
or backward complete Finsler manifold and let $f\colon M\to \R$ be a
$C^2$ convex  function having a non-degenerate critical point $p_0$.
Then $p_0$ is a global minimum  point for $f$ and it is the unique
critical point of $f$. \epr
\begin{proof}
Let $\gamma\colon[a,b]\to M$ be a non-constant geodesic starting at
$\gamma(a)=p_0$. Then the function $g(s)=f(\gamma(s))$ is convex in
$[a,b]$, that is, $g''(s)\geq 0$ for any $s\in[a,b]$. As $p_0$ is a
non-degenerate  critical point, $g'(a)=0$ and $g''(a)>0$. Clearly
$g'$ is an increasing function in $[a,b]$, so that $g'\geq 0$.
Assume that there exists a point $s_0\in]a,b]$ such that
$g'(s_0)=0$, then $g'(s)=0$ for any $s\in[a,s_0]$, which is in
contradiction with $g''(a)>0$. Therefore $g'(s)>0$ for every $s\in
[a,b]$, hence $f(p_0)<f(\gamma(s))$ for any $s$. Now, let $q\in M$
an arbitrarily chosen point of $M$, by the Finslerian Hopf-Rinow
theorem there exists a geodesic $\gamma_q \colon[a,b]\to M$ such
that $\gamma_q(a)=p_0$ and $\gamma_q(b)=q$. Since we have shown that
$f(p_0)<f(\gamma_q(b))=f(q)$ and $g'(b)>0$,  it follows that $p_0$ is
a global minimum and the function $f$ does not admit other critical
points.
\end{proof}
\bpr \label{prop:ilimitado} Let $f\colon M\to \R$ be a $C^2$ convex
function of a  forward or backward complete Finsler manifold $(M,F)$
and suppose that there exists a non-degenerate critical point $p_0$
of $f$ (unique by Proposition \ref{prop:minimum}). Then
\[\lim_{d(p_0,x)\to \infty}f(x)=+\infty,\quad\quad\text{and}\quad\quad
\lim_{d(x,p_0)\to \infty}f(x)=+\infty.\]
\epr
\begin{proof}
By Proposition \ref{prop:minimum} if $p_0$ is a non-degenerate
critical point of $f$, then it is a global minimum and the unique
critical point of $f$. We prove now that for any diverging sequence
$(x_n)_{n\in{\mathbb{N}}}$ in $M$, it holds that $\lim_{n\to
\infty}f(x_n)=+\infty$.

We assume that the Finsler manifold is forward complete (in the
backward completeness case the proof is analogous).

Let $A$ be a normal neighborhood of $p_0$, i.e., there exists a
star-shaped open neighborhood $U$ of zero $U\subset T_{p_0}M$ such
that $\exp_{p_0}\colon U\to A$ is a diffeomorphism of class $C^1$ in
$U$ and $C^\infty$ in $U\setminus\{0\}$ (see \cite[\S
5.3]{BaChSh00}). For all $x\in A\setminus\{p_0\}$, set
$u(x)=\exp_{p_0}^{-1}(x)/F(p_0,\exp_{p_0}^{-1}(x))$,  define
$\gamma_x\colon]0,F(p_0,\exp_{p_0}^{-1}(x))]\to M$ as
$\gamma_x(s)=\exp_{p_0} (su(x))$ and
\[\phi^+(x)=\left.\frac{\de}{\de s}\right|_{s=F(p_0,\exp_{p_0}^{-1}(x))}f(\gamma_x(s)).\]
Clearly $\phi^+$ is $C^2$ in $A\setminus\{p_0\}$ and non-negative, because
convex functions have increasing derivative.
Now fix  $r\in\R$
small enough such that the sphere $S^+_r(p_0)=\{y\in M\ |\ \dist(p_0,x)=r\}$ is  contained in $A$, and define
\[\delta^+_0=\min_{x\in S^+_r(p_0)}\phi^+(x).\]
 The number
$\delta^+_0$  is positive, because otherwise there would exist $x\in
S^+_r(p_0)$ such that $\phi^+(x)=0$ and $f\circ\gamma_x$ would be a
constant function, in contradiction with the hypothesis that $p_0$
is non-degenerate.

Now,  set
\[
f^+_0=\min_{x\in S^+_r(p_0)}f(x)>-\infty.
\]
Moreover, let $\gamma_n\colon[0,b_n]\to M$ be any minimal geodesic
from $p_0$ to $x_n$ having constant speed equal to one, and let
$g_n:[0,b_n]\rightarrow \R$  be defined as $g_n(s)=f(\gamma_n(s))$.
Since the sequence $(x_n)_{n\in{\mathbb{N}}}$ diverges, we can suppose
that $b_n>r$, so that $p_n=\gamma_n(r)$ is well-defined. By the
convexity of $f$, we get
\begin{align*}
f(x_n)=g_n(b_n)&\geq g_n(r)+g'_n(r)(b_n-r)\\
&=f(p_n)+\phi^+(\gamma_n(r)) (b_n-r)\\
&\geq f^+_0+\delta^+_0(\dist (p_0,x_n)-r)\stackrel{n}{\longrightarrow}+\infty.
\end{align*}
We can analogously also prove that 
$\lim_{\mathrm{d}(x,p_0)\to-\infty}f(x)=+\infty$, by considering the
minimizing unit speed geodesic $\gamma_n\colon[c_n,0]\to M$ from
$x_n$ to $p_0$,  the {\em backward exponential} and a compact
backward sphere $S^-_r(p_0)$ contained in the image of the domain of
the backward exponential. The backward exponential is defined as
$\exp^-_{p_0}(v)=\gamma_v(-1)$, where $\gamma_v$ is the unique
constant speed geodesic with $\gamma_v(0)=p_0$ and $\dot
\gamma_v(0)=v$ and
\[\phi^-(x)=\left.\frac{\de}{\de s}\right|_{s=-F\big(p_0,(\exp^-_{p_0})^{-1}(x)\big)}f(\gamma_x(s)).\]
where  $\gamma_x\colon[-F\big(p_0,(\exp^-_{p_0})^{-1}(x)\big),0]\to M$ is now given by $\gamma_x(s)=\exp^-_{p_0}(s u(x))$
and $u(x)=-(\exp^-_{p_0})^{-1}(x)/F\big(p_0,(\exp^-_{p_0})^{-1}(x)\big)$.
Thus we have
\begin{align*}
f(x_n)=g_n(c_n)&\geq g_n(-r)+g'_n(-r)(c_n+r)\\
&=f(p_n)+\phi^-(\gamma_n(-r)) (c_n+r)\\
&=f(p_n)-\phi^-(\gamma_n(-r)) \big(\dist(x_n,p_0)-r)\\
&\geq f^-_0-\delta^-_0(\dist (x_n,p_0)-r)\stackrel{n}{\longrightarrow}+\infty,
\end{align*}
where $f^-_0$ is the minimum value of $f$ on $S^-_r(p_0)$ and $\delta^-_0<0$ is the maximum value of $\phi^-$ on $S^-_r(p_0)$.
\end{proof}
\bere
Similar to  \cite[Proposition 2.5 and Lemma
2.6]{BisO'N69},  it can be proved that if a manifold $M$ admits a
$C^1$ function $f$ with locally Lipschitz differential, having a
unique critical point which is a global minimum and having compact
sublevels $f^c=\{x\in M\ |\ f(x)\leq c\}$, $c\in \mathbb{R}$, then
it is contractible. So a forward or backward complete Finsler
manifold  admitting  a $C^2$ convex function  having a
non-degenerate minimum point is contractible (observe that in this
case the sublevels of $f$ are compact as a consequence of the
Finslerian Hopf-Rinow theorem and Proposition \ref{prop:ilimitado}).
Apart from those in \cite{BisO'N69}, other results about the
topological and differentiable structure of a Riemannian manifold
endowed with a (non-necessarily $C^2$) convex function can be found
in \cite{bang, GreShi81, GreShi81a}. \ere

\bt\label{geonfinito} Let $(M,F)$ be a forward or backward complete
Finsler manifold that admits a $C^2$ function $f\colon M\to\R$ with
positive definite Hessian $H_f$ everywhere and that has a minimum
point. If $p$ and $q$ are non-conjugate points of $M$, then the
number of geodesics in $M$ joining $p$ and $q$ is finite. \et
\begin{proof}
We begin by showing that  a compact subset $C\subset M$ containing
the image of every geodesic joining $p$ and $q$ does exist. Indeed,
set
\[d=\max \{f(p),f(q)\},\]
since convex functions reach the maximum at the endpoints of the
interval, it follows that $f(\gamma(s))\leq d$ and $\gamma
([0,1])\subseteq f^d$. By Proposition \ref{prop:ilimitado} and the
Finslerian Hopf-Rinow theorem, the subset $C=f^d$ is compact.

We now  claim that there exists a constant $E_0$ such that
\beq\label{e0} F(\gamma ,\dot \gamma)\leq E_0, \eeq for every
geodesic $\gamma \colon[0,1]\to M$  connecting $p$ to $q$. To prove
it by contradiction, let us  assume that there exists a sequence of
geodesics $\gamma_n\colon[0,1]\to M$ joining $p$ and $q$  and having
constant speed $E_n$  with $E_n\to+\infty$, as $n\to \infty$.
Consider  the speed one geodesics $y_n\colon[0,E_n]\to M$ given by
$y_n(s)=\gamma_n(s/E_n)$. The sequence of vectors $\{\dot
y_n(0)\}\subset T_pM$  admits  a subsequence converging to $v\in
T_pM$. Moreover, as the images of the curves $y_n$ are contained in
$C$, the image of the geodesic $y$, such that $y(0)=p$ and $\dot
y(0)=v$, is  also contained in $C$. Since $H_f$ is positive
definite, there exists a constant $\lambda_0=\lambda_0(C)>0$ such
that, for all $p\in C$ and all $v\in T_pM$, the following holds:
\[(H_f)_{(p,v)}(v,v)\geq \lambda_0F^2(p,v).\]
So if we set $\rho(s)=f(y(s))$, then
\[\rho''(s)=(H_f)_{(y(s),\dot y(s))}(\dot y(s),\dot y(s))\geq \lambda_0F^2(y(s),\dot y(s))=\lambda_0>0,\]
for every $s\in [0,+\infty)$ and hence $\lim_{s\to
\infty}\rho(s)=+\infty$, which is in contradiction to the fact that
the image of $y$ is contained in the compact set $C$.

Now we can conclude the proof observing that if there exists an
infinite number of geodesics connecting $p$ to $q$, we can consider
a sequence of such  geodesics $\gamma_m\colon [0,1]\to M$,  having
initial vectors $\dot\gamma_m(0)$. From \eqref{e0}, the sequence
$\dot\gamma_m(0)$ is contained in a compact subset of $T_pM$; hence,
it converges, up to pass to a subsequence, to a vector $v\in T_pM$.
Then, by a standard argument on the continuous dependence of
solutions of ODEs with respect to initial  data,  the geodesics
$\gamma_m$ uniformly converge to the geodesic $\gamma\colon [0,1]\to
M$ satisfying the initial conditions $\gamma(0)=p,\dot\gamma(0)=v$.
By uniform convergence, we also have  $\gamma(1)=q$. Thus, the
Finslerian exponential map $\exp_p$ is not injective in a
neighborhood of $v$, in contradiction with the fact that $p$ and $q$
are two non-conjugate points (see \cite[Prop. 7.1.1]{BaChSh00}).
\end{proof}
\bere  It is well known that if a manifold $M$ is
non-contractible in itself (for instance $M$ is compact),  then for any
Finsler metric $F$ on the manifold $M$, there exist infinitely many
geodesics joining two arbitrary points $p$ and $q$ of $M$, see
\cite{{CJM08}}. On the other hand, as in the Riemannian case, there
are circumstances in which the number of geodesics connecting any
two points on $(M,F)$ is exactly equal to $1$. For instance, the
Cartan-Hadamard Theorem holds for forward complete Finsler manifolds
having non positive flag curvature; thus if $M$ is simply connected,
the exponential map is a $C^1$ diffeomorphism from the tangent space
at any point of $M$ onto $M$ (see \cite[Theorem 9.4.1]{BaChSh00}).
Under the assumptions of Theorem~\ref{geonfinito},  the
existence of infinitely many geodesics is excluded, but the existence of multiple
geodesics between two points is allowed. This fact seems to be
interesting in the gravitational lens effect, where a multiplicity
of light rays occurs between an
observer and the world line of a source, see \cite{GMPJMP,PE}. \ere
\end{section}
\begin{section}{Applications}\label{appli}
\begin{subsection}{Randers metrics}\label{randerssection}
Let $(M,h)$ be a Riemannian manifold and  let $\omega$ be a one form on
$M$ such that for any $x \in M$,
\begin{equation}\label{omegaminusone}
\|\omega\|_x = \sup_{v\in T_x M\setminus
0}\frac{|\omega(v)|}{\sqrt{h(v,v)}} < 1.
\end{equation}
Then the Randers metric associated with $h$ and $\omega$ is the
Finsler metric $F$ on $M$ defined as \beq\label{randers}
F(x,y)=\sqrt{h(y,y)}+\omega(y). \eeq The couple $(M,F)$ with $F$
given by (\ref{randers}) is called \textit{Randers manifold}. Let us
observe that the condition $\|\omega\|_x<1$, for all $x\in M$,
implies not only that $F$ is positive, but also that it has
fiberwise strongly convex square (see  \cite[\S 11.1]{BaChSh00}).

Such type of Finsler metrics, with $h$ Lorentzian,  were considered in
1941 by G. Randers  in a paper (see \cite{Rander41}) about the
equivalence of relativistic electromagnetic theory (where the four
dimensional space-time is endowed with  a metric of the form
\eqref{randers})  and the five-dimensional Kaluza-Klein
theory.
\begin{remark}\label{completo}
As observed in \cite[Remark 4.1]{CJM08}, if the Riemannian metric
$(M,h)$ is complete and \beq\label{superscemo} \|\omega\| =
\sup_{x\in M}\|\omega\|_x < 1, \eeq then the Randers manifold
$(M,F)$ is forward  and backward  complete.
\end{remark}
By using the Levi-Civita  connection \nablag of the  metric $h$, the
geodesic equations of a Randers metric,  parametrized to have
constant Riemannian speed, can be written as  (see \cite[p.
297]{BaChSh00}) \beq\label{geomagRiem} \nablag_{\dot\sigma}\dot
\sigma=\sqrt{h(\dot\sigma,\dot\sigma)}\hat \Omega(\dot\sigma), \eeq
where $\hat \Omega$ is the $(1,1)$-tensor field  metrically
equivalent to $\Omega=\de \omega$, i.e for every $(x,v)\in TM$,
$\Omega(\cdot,v)=h(\cdot,\hat \Omega(v))$. We observe that if we
define a vector field $B$ such that $\omega(v)=h(B,v)$, then Eq.
\eqref{geomagRiem} can be expressed as
\[
\nablag_{\dot\sigma}\dot \sigma=\sqrt{h(\dot\sigma,\dot\sigma)} {\rm Curl} B(\dot\sigma),
\]
where ${\rm Curl} B(v)$ is the vector satisfying
\[h({\rm Curl} B(v),w)=h(\nabla^h_w B,v)-h(\nabla^h_v B,w)\]
for every $v,w$ in $T_xM$.

The existence of convex functions is known for several classes of
Riemannian manifolds. For instance, let $M = \mathbb{R}^N$ and let
$h_0$ be the standard Riemannian metric on $\mathbb{R}^N$ and
consider a conformally equivalent metric $h$ to $h_0$, so there
exists a smooth, positive function $\eta\colon
\mathbb{R}^N\to\mathbb{R}$ such that $h=\eta(x) h_0$. Then, if
$$\eta(x)-\frac{3}{2} |\nabla \eta|\cdot|x|>0,$$
where $|\cdot|$ denotes the Euclidean norm, then the function
$G(x)=|x|^2$ is strictly convex for the conformal metric $h$ (see
\cite[Lemma 3.1]{GiMaPi99}). Moreover, if $(M,h)$  is a complete
non-compact manifold having non-negative sectional curvature, the
Busemann function with changed sign  is convex (see
\cite{CheGro72}). Finally, on a simply connected complete Riemannian
manifold with non-positive sectional curvature the smooth function
$x\to \left(\distg (x_0,x)\right)^2$,  $x_0\in M$,  is strictly
convex (see \cite{BisO'N69}).

Let $(M,F)$ be a Randers manifold, with $F$ given by
(\ref{randers}). Our aim is to give conditions on the associated
vector field $B$ and on the covariant differential $\nabla^h B$
ensuring that a convex function with respect to the metric
Riemannian $h$ is still convex with respect to the Randers metric
$F$. To this end, we need to write the equation satisfied by a
geodesic, parametrized with constant Randers speed, using the
Levi-Civita connection $\nabla^h$ and not the Chern
connection. Since geodesics joining two fixed points are the
critical points of the length functional (with respect to the
Randers metric $F$)
\[
L(\gamma) = \int_0^1
\left[\sqrt{h(\dot\gamma,\dot\gamma)}+\omega(\dot\gamma)\right]{\rm
d}s,
\]
they satisfy the Euler-Lagrange equations
\begin{equation*}
\nablag_{\dot\sigma}\left(\dot \sigma/\sqrt{h(\dot\sigma,\dot\sigma)}\right)=\hat \Omega(\dot\sigma)
\end{equation*}
and, after some straightforward computations, we obtain
\begin{equation*} -\frac{\frac{\de}{\de
s}(\sqrt{h(\dot\sigma,\dot\sigma)})}{\sqrt{h(\dot\sigma,\dot\sigma)}}\dot
\sigma+\nablag_{\dot\sigma}\dot
\sigma=\sqrt{h(\dot\sigma,\dot\sigma)}\hat \Omega(\dot\sigma).
\end{equation*}
If $\sigma$ is parametrized with constant Randers speed
$\sqrt{h(\dot\sigma,\dot\sigma)}+\omega(\dot\sigma)$, we can replace
the term $\frac{\de}{\de s}(\sqrt{h(\dot\sigma,\dot\sigma)})$ in the
last equation by the term $-\frac{\de}{\de s}(\omega(\dot\sigma))$,
obtaining \beq\label{geomag2} \nablag_{\dot\sigma}\dot
\sigma=\sqrt{h(\dot\sigma,\dot\sigma)}\hat
\Omega(\dot\sigma)-\frac{\frac{\de}{\de
s}\left(\omega(\dot\sigma)\right)}{\sqrt{h(\dot\sigma,\dot\sigma)}}\dot\sigma.
\eeq

In the next proposition we compute $(H_f)(y,y)$, for each $(x,y)\in TM\setminus 0$,
using the  Levi-Civita connection of $h$.
\bpr
Let $f\colon M\to\R$ be a $C^2$ function. For each $(x,y)\in TM\setminus 0$ we have
\bml\label{hf}
H_f(y,y)=H^h_f(y,y)+\sqrt{h(y,y)}\,\,h\big(\nabla^h f,{\rm Curl}B(y)\big)\\-\frac{h(\nabla^h f,y)}{F(x,y)}\left(h(\nabla^h_yB,y)+\sqrt{h(y,y)}\,\,h(B,{\rm Curl}B(y))\right),
\eml
where  $\nabla^hf$ and $H^h_f$ denote, respectively, the gradient and the Hessian of $f$ with respect to the metric $h$.
\epr
\begin{proof}
Let $\sigma$ be a geodesic of $(M,F)$  parametrized with constant
Randers speed and such that $\sigma(0)=x$, $\dot\sigma(0)=y$. We set
$\rho(s)=f(\sigma(s))$.  From Eq. \eqref{geomag2},  recalling that
for any $(x,v)\in TM$, $\hat\Omega(v)=\mathrm{Curl}B(v)$,  we get
\begin{align}\label{deriv}
\rho''(s)=&H^h_f(\dot\sigma,\dot\sigma)+h(\nablag f,\nablag_{\dot\sigma}\dot\sigma)\nonumber\\
=&H^h_f(\dot\sigma,\dot\sigma)+h(\nablag f,\sqrt{h(\dot\sigma,\dot\sigma)}\,\,{\rm Curl}B(\dot\sigma))-\frac{\frac{\de}{\de s}\left(h(B,\dot\sigma)\right)}{\sqrt{h(\dot\sigma,\dot\sigma)}}h(\nablag f,\dot\sigma).
\end{align}
Now observe that  from \eqref{geomag2}, we get \beq \frac{\de}{\de
s}\left(h(B,\dot\sigma)\right)=
\frac{\sqrt{h(\dot\sigma,\dot\sigma)}}{F(\sigma,\dot\sigma)}\big(h(\nabla^h_{\dot\sigma}
B,\dot\sigma)+ \sqrt{h(\dot\sigma,\dot\sigma)}h(B,{\rm
Curl}B(\dot\sigma))\big). \label{due} \eeq Substituting \eqref{due}
into \eqref{deriv} we obtain \eqref{hf}.
\end{proof}
Now we  give a condition which ensures that a $h$-convex $C^2$ function
is  also convex with respect to $F$. We denote by $|\cdot|$ the norm
with respect to  the Riemannian metric $h$ and by  $\|\cdot\|$ the
corresponding norms for tensor fields  on $M$.
\bpr\label{convessoRanders} Let $f\colon M\to \R$ be a $C^2$
function which is convex with respect to the Riemannian metric $h$.
Assume that $f$ has a strictly positive Riemannian Hessian
$H^h_f=\nablag(\de f)$, i. e. there exists a strictly positive
function  $\lambda\colon M\to\R$ such that
$H^h_f(v,v)\geq\lambda(x)|v|^2$, for all $(x,v)\in TM$. Moreover
assume that
\[
3\|\de f\|\|\nabla^h B\|/(1-|B|)<\lambda(x).
\]
Then $f$ is strictly convex with respect to the Randers  metric $F$.
\epr
\begin{proof}
We have
\[\big|h(\nablag f,\sqrt{h(y,y)}{\rm Curl}B(y))\big|\leq 2\|\nabla^hB\|\|\de f\||y|^2\]
and
\baln
\lefteqn{
\left|\frac{h(\nabla^h f,y)}{F(x,y)}\left(h(\nabla^h_yB,y)+\sqrt{h(y,y)}h(B,{\rm Curl}B(y))\right)\right|}&\\
 &\leq\frac{1}{(1-|B|)|y|}\left(\|\nabla^hB\||y|^2+2\|\nabla^hB\||B||y|^2\right)\|\de f\||y|\\
&=\|\nabla^hB\|\frac{1+2|B|}{1-|B|}
\|\de f\||y|^2.
\ealn
Thus, from \eqref{hf} we obtain
\begin{equation*}
H_f(y,y)\geq \left(\lambda (x)-\frac{3\|\de f\|\|\nabla
B\|}{1-|B|}\right)|y|^2> 0.
\end{equation*}
\end{proof}
From the above proposition, Theorem~\ref{geonfinito} and
Remark~\ref{completo}, the following proposition also holds.
\bpr\label{finitezzaRanders} Let $M$ be a smooth manifold and let
$F$ be a Randers metric on $M$ satisfying \eqref{superscemo} and
assume that the Riemannian metric $h$ on $M$ is complete. Assume
that there exists a $C^2$ function $f\colon M\to \R$ having a
minimum point and Hessian $H^h_f$ satisfying
\[
H^h_f(v,v)\geq\lambda(x)|v|^2,\] for some positive function
$\lambda\colon M\to\R$ and for any $(x,v)\in TM$. If
\[\|\de
f\|\|\nabla B\|/(1-|B|)<\lambda(x),\]
then, for any couple $p$ and
$q$ of non-conjugate points for $(M,F)$, there exists only a
finite number of geodesics connecting $p$ to $q$ with respect to the
Randers metric $F$. \epr \bere

The hypothesis that  the points $x_0$ and $x_1$ are non-conjugate is
a reasonable assumption to have only a finite number of geodesics
between two points on a Riemannian or a Finsler manifold (for
example it forbids the existence of a continuum of geodesics with
endpoints $x_0$ and $x_1$). Anyway it is not a necessary condition.
Indeed on a Randers manifold, by using {\em Stationary-to-Randers
correspondence} \cite{CJS} (see also  next subsection) and some
bifurcation results for lightlike geodesics in a Lorentzian manifold
(see \cite[Proposition 13]{javpic}), it can be proved that if $x_0$
and $x_1$ are conjugate along the geodesic $\gamma$,
$\gamma\colon[0,1]\to M$, $\gamma(0)=x_0$, $\gamma(1)=x_1$, then
there  exists a continuum $(\gamma_{\eps})_{\eps\in [0,\eps_0)}$ of
geodesics, $\gamma_{\eps}\colon[0,a]\to M$, $a>1$, and a function
$s=s(\eps)\colon[0,\eps_0)\to [0,a]$ such that
$\gamma_{\eps}(0)=x_0$, for each $\eps\in [0,\eps_0)$,  $s(\eps)\to
1$, $\dot\gamma_{\eps}(0)\to \dot\gamma(0)$,  as $\eps\to 0$,  and
$\gamma_{\eps}(s(\eps))=\gamma(s(\eps))$.

Moreover, by Sard's Theorem and the fact that conjugate points  are
critical values of the exponential map, we know that the set of
non-conjugate points to a given point $x_0$ is generic in $M$. Again using
Stationary-to-Randers correspondence and a recent result about
genericity of the condition for being a point and a line in a
standard stationary spacetime non-conjugate (see \cite{ggp,gj}), we
have that the set of all the $C^2$ Riemannian metrics $h$ and the
set of all the $C^2$ one-forms $\omega$ on $M$, for which two fixed
distinct  points $x_0,\ x_1\in M$ are non-conjugate in the Randers
manifold $(M,\sqrt{h}+\omega)$, are generic in  the sets of all
the bilinear  forms  and all the one-forms on $M$,  with respect to
a suitable topology (in particular, such a topology implies
$C^2$-convergence on compact subsets of $M$).  Finally we mention that a
systematic study of the Finslerian cut locus can be found in \cite{Hassan}. \ere
\end{subsection}
\begin{subsection}{Applications to stationary spacetimes}\label{stationaryspacetimes}
In this subsection, we apply the results in Section~\ref{finito} to
the study of  causal geodesics connecting a point to a timelike
curve on a standard stationary Lorentzian  manifold.

A {\em standard stationary spacetime} is a Lorentzian manifold
$(L,l)$, where $L$ splits as a product $L=M\times \R$, $M$ is
endowed with a Riemannian metric $g_0$, and there  exist a vector
field $\delta$ and a positive function $\beta$ on $M$ such that the
Lorentzian metric $l$ on $L$ is given by
\begin{equation}\label{l}
l((y,\tau),(y,\tau))=g_0(y,y)+2g_0(\delta,y)\tau -\beta(x)\tau^2,
\end{equation}
for any $(x,t)\in M\times \R$ and  $(y,\tau)\in T_xM\times\R$. We
observe that a stationary spacetime, that is, a Lorentzian manifold
which  admits a timelike Killing field, is standard whenever the
timelike Killing field is complete and the spacetime is
distinguishing (see \cite{JS08}).

A curve  $(x(s),t(s))$ in $L$ is a future-pointing lightlike
geodesic if and only if $x$ is a geodesic for the Randers metric,
that we call {\em Fermat metric},  defined as \bal
&F(x,y)=\sqrt{p(\delta,y )^2+p(y
,y)}+p(\delta,y),\label{RandersRiemannian} \eal where
$p=\frac{1}{\beta}g_0$, parametrized with constant Riemannian speed
$p(\dot x,\dot x)+p(\delta, \dot x)^2$, and $t$ coincides, up to a
constant, with the Fermat length of $x$  (see \cite[Theorem
4.5]{CJM08}).

We  need to express the equation satisfied by Fermat geodesics using
the Levi-Civita connection of the metric $p$. For this reason we
denote by $|\cdot|_0$ the norm with respect to the Riemannian metric
$g_0$ and by $\nabla$ the Levi-Civita connection of $g_0$ or the
gradient with respect to $g_0$,  $|\cdot|_1$ and $\tilde\nabla$
denote the norm and the Levi-Civita connection of $p$,  while
$\|\cdot\|_0$ and $\|\cdot\|_1$  denote the corresponding norms of
the tensor fields on $M$. Moreover, in this subsection we set
$h(\cdot,\cdot)=p(\delta,\cdot)^2+p(\cdot,\cdot)$.
\begin{lemma}
A curve $\gamma$ in $(M,F)$, $F$ defined in
\eqref{RandersRiemannian}, parametrized  with   constant Riemannian
speed $ h(\dot\gamma,\dot\gamma)$, is a geodesic of $(M,F)$   if and
only if it satisfies the equation \beq\label{nablap} \tilde
\nabla_{\dot\gamma}\dot\gamma=F(\gamma,\dot\gamma)\tilde\Omega(\dot\gamma)-\frac{\de}{\de
s}\big(p(\delta,\dot\gamma)\big)\delta, \eeq where
$\tilde\Omega(y)=\tilde\nabla^*\delta(y)-\tilde\nabla\delta(y)$,
$\tilde\nabla \delta(y)=\tilde\nabla_{y}\delta$ and
$\tilde\nabla^*\delta$ is the adjoint with respect to $p$ of
$\tilde\nabla\delta$.
\end{lemma}
\begin{proof}
Consider the length functional of the Fermat metric
\begin{equation}\label{ELFermat}
L(x) = \int_0^1\left[ \sqrt{p(\delta(x),\dot x)^2 + p(\dot x,\dot
x)} + p(\delta(x),\dot x)\right]{\rm d}s.
\end{equation}
Let $\tilde\nabla$ be the Levi-Civita connection of the Riemannian
metric $p$; the Euler-Lagrange equations of the
functional \eqref{ELFermat} can
be written as
\begin{equation}\label{eulerlagrange}
-\tilde\nabla_{\dot\gamma}\left(\frac{\dot\gamma+p(\delta,\dot\gamma)\delta}{\sqrt{h(\dot\gamma,\dot\gamma)}}\right)+
\frac{p(\delta,\dot\gamma)\tilde \nabla^*\delta(\dot\gamma)}{\sqrt{h(\dot\gamma,\dot\gamma)}}+
\tilde \nabla^*\delta(\dot\gamma)-\tilde \nabla\delta(\dot\gamma)=0.
\end{equation}
Hence, if $\gamma$ is parametrized to have  constant Riemannian speed,  we get:
\baln
\tilde\nabla_{\dot\gamma}\dot\gamma
&=-\tilde\nabla_{\dot\gamma}\big(p(\delta,\dot\gamma)\delta\big)+p(\delta,\dot\gamma)\tilde \nabla^*\delta(\dot\gamma)+\sqrt{h(\dot\gamma,\dot\gamma)}\left(\tilde\nabla^*\delta(\dot\gamma)-\tilde \nabla\delta(\dot\gamma)\right)\\
&=-\frac{\de}{\de s}\big(p(\delta,\dot\gamma)\big)\delta+p(\delta,\dot\gamma)\left(\tilde \nabla^*\delta(\dot\gamma)-\tilde \nabla\delta(\dot\gamma)\right)\\
&\quad\quad+\sqrt{h(\dot\gamma,\dot\gamma)}\left(\tilde\nabla^*\delta(\dot\gamma)-\tilde \nabla\delta(\dot\gamma)\right)\\
&=F(\gamma,\dot\gamma)\tilde\Omega(\dot\gamma)-\frac{\de}{\de
s}\left(p(\delta,\dot\gamma\right)\delta. \ealn

By computing $p(\tilde\nabla_{\dot\gamma}\dot\gamma,\dot\gamma)$ we can easily  see  that any solution of \eqref{nablap} has constant $h$-Riemannian speed and
it satisfies equation \eqref{eulerlagrange}.
\end{proof}
\begin{lemma}
A geodesic $\sigma$ of $(M,F)$, $F$ defined in
\eqref{RandersRiemannian}, parametrized  with   constant Randers
speed $\sqrt{h(\dot\sigma,\dot\sigma)}+p(\delta,\dot\sigma)$
satisfies  the equation \beq\label{nablapRanders} \tilde
\nabla_{\dot\sigma}\dot\sigma=F(\sigma,\dot\sigma)\tilde\Omega(\dot\sigma)-
\frac{\frac{\de}{\de
s}\left(p(\delta,\dot\sigma)\right)}{\sqrt{h(\dot\sigma,\dot\sigma)}}\left(\dot\sigma+F(\sigma,\dot\sigma)\delta\right).
\eeq
\end{lemma}
\begin{proof}
Since $\sigma$ has constant Randers speed, we have that
$$\frac{\de}{\de s}\sqrt{h(\dot\sigma,\dot\sigma)}=-\frac{\de}{\de
s}(p(\delta,\dot\sigma)).$$  Using this  equality in
\eqref{eulerlagrange},  we obtain \eqref{nablapRanders}. Computing
$p(\tilde\nabla_{\dot\sigma}\dot\sigma,\dot\sigma)$, we deduce that
the solutions of \eqref{nablapRanders} have constant Randers speed,
so that they are solutions of the Euler-Lagrange equations
\eqref{eulerlagrange}.
\end{proof}
We observe that a link between the geodesics of a Randers metric and
those of a stationary spacetime also exists for timelike geodesics
of $(L,l)$. Indeed, each timelike geodesic of $(L,l)$, $l$ as in
\eqref{l}, can be  seen as the projection on $L$ of a lightlike
geodesic in the stationary spacetime $ (\tilde L,\tilde l)$, where
$\tilde L=M\times\R\times\R$ and
\begin{equation}\label{extend}\tilde l((y,v,\tau),(y,v,\tau))=g_0(y,y)+v^2 + 2g_0(\delta,y)\tau -\beta(x)\tau^2.
\end{equation}
More precisely, as it was observed in \cite[Section 4.3]{CJM08}, a
curve $z(s)  = (x(s),u(s),t(s))$ in $\tilde L$ is a  lightlike
geodesic if and only if $(x(s),t(s))$ is a timelike geodesic of
$(L,l)$ and $\dot u(s)$ is  constant and equal to $E$, where
$-E^2=l\big((\dot x(s),\dot t(s)),(\dot x(s),\dot t(s))\big)$.  We
recall that a timelike geodesic is parametrized with respect to {\em
proper time} if $E=1$. As a  consequence, the existence of timelike
geodesics with arrival proper time equal to a given $T>0$ and
joining a point $(x_0,\varrho_0)$ to a timelike curve
$\ell(\varrho)=(x_1,\varrho)$ can be deduced  from the existence of
geodesics connecting $(x_0,0)$ to $(x_1,T)$  on the manifold
$N=M\times \R$ endowed with the Fermat metric $\tilde F$, where
$\tilde F$ is given by \beq\label{RandersforE} \tilde
F((x,u),(y,v))=\sqrt{\frac{1}{\beta(x)}(g_0(y,y)+v^2)+\frac{1}{\beta(x)^2}g_0(\delta,y)^2
}+\frac{1}{\beta(x)}g_0(\delta,y), \eeq for all $((x,u),(y,v))\in
TN$. Indeed, a curve $(x,t): [0,T]\rightarrow L$ is a future-pointing
timelike geodesic of $(L,l)$, parametrized with respect to proper
time,  if and only if $[0,T]\ni s\rightarrow(x(s), u(s),t(s))\in
\tilde L$ is a lightlike geodesic (and therefore $u(s)=s$ up to an
initial constant).  At the same time, this fact is equivalent to the
requirement that the curve $[0,T]\ni s\rightarrow(x(s), u(s))\in N$
is  a geodesic of $(N,\tilde F)$,  parametrized with constant
Riemannian speed \footnote{Notice that any regular curve
$\gamma\colon [0,T]\to M$ in a Randers manifold $(M,F)$, parametrized
with constant Randers speed, can be parametrized on the same
interval $[0,T]$ with constant Riemannian speed.}  and $t=t(s)$
equal, up to an additive  constant, to the length with respect to
$\tilde F$ of the curve $(x(r),u(r))$, $r\in [0, s]$.

Before stating the main result of this section, we need the
following definition:
\bd
Let $(L,l)$ be a Lorentzian manifold,
$p\in L$ and $\ell\colon(a,b)\to L$ a timelike curve such that
$p\not\in\ell((a,b))$. We say that $p$ and $\gamma$ are {\em future
lightlike} (resp. {\em $T$-timelike}) {\em non-conjugate,}  if  the
points $p$ and $\gamma(1)$ (resp. $\gamma(T)$) are non-conjugate
along $\gamma$, for all the future-pointing lightlike  geodesics
$\gamma\colon[0,1]\to L$ (resp. timelike geodesics $\gamma\colon
[0,T]\to L$ parametrized with respect to proper time) such that
$\gamma(0)=p$ and $\gamma(1)\in \ell((a,b))$ (resp.
$\gamma(T)\in\ell((a,b))$). \ed \bere\label{nonconj} It can be
proved (see Theorem~3.2 of \cite{CJM09}) that  if $(L,l)$ is a
standard stationary spacetime then a point $(x_0,\varrho_0)\in L$
and  a curve $\ell(\varrho)=(x_1,\varrho)$, with $x_0,x_1\in M$,
$x_0\neq x_1$, are future lightlike non-conjugate if and only if
$x_0$ and $x_1$ are non-conjugate in the Randers manifold $(M,F)$. \ere
\bere\label{nonconj2} Analogously, if the point   $(x_0,\varrho_0)$
and the curve $\ell$ are future $T$-timelike non-conjugate, then the
points $(x_0,0)$ and $(x_1,T)$ are non-conjugate in $(N,\tilde F)$.
This can be seen by using the extended stationary spacetime $(\tilde
L, \tilde l)$ and the associated Randers  manifold $(N, \tilde F)$.
Indeed, by the fact that any Jacobi vector field along a geodesic in
$(\tilde L, \tilde l)$ has $u$ component which is an affine
function, if the points $(x(0),t(0))=(x_0,\varrho_0)$ and
$(x(T),t(T))\in\ell(\R)$ are non-conjugate along any  timelike
geodesic $s\in[0,T]\mapsto (x(s),t(s))$ in $(L,l)$  parametrized
with respect to proper time and connecting them, then the points
$(x(0),0,t(0))$ and $(x(T),T,t(T))$ are non-conjugate along any
lightlike geodesic $(x(s),s,t(s))$ in $(\tilde L,\tilde l)$
connecting them.  Therefore,  by \cite[Theorem~3.2]{CJM09}, the
points $(x_0,0)$ and $(x_1,T)$ are non-conjugate in $(N,\tilde F)$.
\ere Next proposition follows the same lines as Proposition 4.7 in
\cite{GiMaPi01}, but we point out that  in the  latter there is an
error in the  hypotheses  that $\delta$ and $\beta$ have to
satisfy. Thus, for the sake of clearness, we  redo  the  proof with
slight changes. In addition, we obtain  a new result about the
finiteness of the number of timelike geodesics parametrized with
respect to proper time on a given interval.
\begin{proposition}\label{improve}
Let $(L,l)$ be a standard stationary Lorentzian manifold, with $l$
as in \eqref{l}. Assume that $(M,g_0)$ admits a $C^2$ convex
function $f\colon M\to\R$  with a minimum point and strictly
positive definite Hessian $H_f^{g_0}$. If
\[
\sup_{x\in M}\frac{|\delta|_0}{\sqrt{\beta(x)}}\leq C,\] for some
$C\in \mathbb{R}^+$, and the functions
$\|\nabla\delta\|_0/\sqrt{\beta(x)}$ and $|\nabla\beta|_0/\beta(x)$
are small enough, then there exists at most a  finite number of
future-pointing lightlike geodesics joining the point
$(x_0,\varrho_0)$ with the curve $ \ell(\varrho)=(x_1,\varrho)$,
$(x_0,\varrho_0)$ and $\ell$ being future lightlike non-conjugate.
Moreover if the point $(x_0,\varrho_0)$ and the curve $\ell$ are
future $T$-timelike non-conjugate, the number of
future-pointing timelike geodesics from $(x_0,\varrho_0)$ to $\ell$
and having arrival proper time equal to $T$ is also at most finite.
\end{proposition}
\begin{proof}
By definition of strictly positive definite Hessian (with respect to
the metric $g_0$),   there exists a function $\lambda_0:M\to
(0,+\infty)$ such that
$$H_f^{g_0}(v,v)\geq\lambda_0(x)|v|_0^2,$$
 for all $x\in M$ and $v\in T_xM$.
Let $x:[a,b]\to M$ be a geodesic of $(M,F)$, $F$ as in
\eqref{RandersRiemannian}, and define $\rho(s)=f(x(s))$. Then  $f$
is  strictly  convex for $F$ if $\rho''(s)>0$ for every geodesic
$x$. We compute $\rho''$ using the Hessian of $f$ with respect to
$g_0$: \beq\label{convessa} \rho''(s)=H_f^{g_0}(\dot x,\dot
x)+g_0(\nabla f,\nabla_{\dot x}\dot x)\geq H^{g_0}_f(\dot x,\dot
x)-|\nabla f|_0|\nabla_{\dot x}\dot x|_0. \eeq From Eq.
\eqref{nablapRanders}, observing that
$$|\dot x+
F(x,\dot x)\delta|_1=F(x,\dot x)\sqrt{1+|\delta|_1^2},$$
 we get
\begin{equation}\label{segundanabla}
|\tilde\nabla_{\dot x}{\dot x}|_1\leq 2 F(x,\dot x)\|\tilde\nabla \delta\|_1|\dot x|_1+\frac{|\frac{\de}{\de s}(p(\dot x,\delta))|}{\sqrt{|\dot x|_1^2+p(\delta,\dot x)^2}}F(x,\dot x)\sqrt{1+|\delta|_1^2}.
\end{equation}
By using Eq. \eqref{nablapRanders} again in
\begin{equation*}
\frac{\de}{\de s}\left(p(\dot x,\delta)\right)=p(\tilde\nabla_{\dot x}\dot x,\delta)+p(\dot x,\tilde\nabla_{\dot x}\delta)
\end{equation*}
we obtain \beq \frac{\de}{\de s}p(\dot x,\delta)=\frac{\sqrt{|\dot
x|_1^2+p(\delta,\dot x)^2}}{F(x,\dot
x)(1+|\delta|_1^2)}\big(F(x,\dot x)p\big(\tilde\Omega(\dot
x),\delta\big)+p(\dot x,\tilde\nabla_{\dot
x}\delta)\big).\label{derivata} \eeq Finally, substituting Equation
\eqref{derivata} into \eqref{segundanabla},  by  $F(x,\dot x)\leq
|\dot x|_1(1+2|\delta|_1)$ and  the  Cauchy-Schwartz  inequality, we
deduce that
\begin{equation}\label{nablafinal}
|\tilde\nabla_{\dot x}{\dot x}|_1\leq |\dot x|_1^2\|\tilde\nabla\delta\|_1 H(|\delta|_1),
\end{equation}
where
$$H(r)=\frac{1}{\sqrt{1+r^2}}\left(2(1+2r)(\sqrt{1+r^2}+r)+1\right).$$
We observe that for every couple of vector fields $X$,$Y$ of the
manifold $M$, it holds that
\begin{equation*}
\tilde\nabla_XY=\nabla_XY+\frac{\beta}{2}((X\beta^{-1})Y+(Y\beta^{-1})X-g_0(X,Y)\nabla \beta^{-1})
\end{equation*}
(see, for example, \cite[p. 181]{Carmo92}). Hence, after some calculations, we  get
\begin{align}\label{preciosa}
|\nabla_{\dot x}\dot x|_0&\leq |\tilde\nabla_{\dot x}\dot x|_0+\frac{3|\nabla\beta|_0}{2\beta(x)}|\dot x|_0^2;&
\|\tilde\nabla\delta\|_0&\leq\|\nabla\delta\|_0+\frac{3}{2}|\delta|_0\frac{|\nabla\beta|_0}{\beta(x)}.
\end{align}
As $\|\tilde\nabla \delta\|_1=\|\tilde\nabla \delta\|_0$ and
$|\tilde\nabla_{\dot x}\dot x|_0=\sqrt{\beta(x)}|\tilde\nabla_{\dot
x}\dot x|_1$, by using inequalities \eqref{nablafinal} and
\eqref{preciosa} we get
\begin{align*}
|\nabla_{\dot x}\dot x|
_0\leq |\dot x|^2_0\left[\left(\frac{\|\nabla\delta\|_0}{\sqrt{\beta(x)}}+\frac{3}{2}\frac{|\delta|_0}{\sqrt{\beta(x)}}\frac{|\nabla \beta|_0}{\beta(x)}\right)H\left(\frac{|\delta|_0}{\sqrt{\beta(x)}}\right)+\frac{3}{2}\frac{|\nabla\beta|_0}{\beta(x)}\right].
\end{align*}
Since we have assumed that  $\sup_{x\in
M}\frac{|\delta|_0}{\sqrt{\beta(x)}}\leq C$,  if
$\frac{\|\nabla\delta\|_0}{\sqrt{\beta(x)}}$ and
$\frac{|\nabla\beta|_0}{\beta(x)}$ are small enough, Eq.
\eqref{convessa} implies that $f$ is a  strictly  convex function.
Moreover, let us observe that the  hypothesis $\sup_{x\in
M}\frac{|\delta|_0}{\sqrt{\beta(x)}}\leq C$, for  some  $C\in
\mathbb{R}$ and completeness of $g_0$, imply forward and backward
completeness of the Fermat metric (see Remark \ref{completo} and
\cite[Remark 4.13, Eq. (47)]{CJM08}).  From Remark~\ref{nonconj},
the points $x_0$ and $x_1$ are non-conjugate in $(M,F)$ and  then by
Theorem~\ref{geonfinito}   we conclude that the number of lightlike
geodesic is finite. In the case of geodesics parametrized with
respect to proper time and having fixed arrival proper time $T$,  we
observe that the metric $\tilde l$ in \eqref{extend} is a stationary
Lorentzian metric with $\tilde \delta=(\delta,0)$ and
$\tilde\beta(x,u)=\beta(x)$, so that we aim to apply the first part
of the theorem to $(\tilde L, \tilde l)$, the point
$(x_0,0,\varrho_0)$ and the line $\R\ni\rho\mapsto
(x_1,T,\rho)\in\tilde L$.  It is clear that the hypotheses on
$\tilde\delta$ and $ \tilde \beta$ are also satisfied in this case.
To show the existence of a convex function, we proceed as follows:
consider a real strictly convex function  $g:\mathbb{R}\to
\mathbb{R}$ having a minimum point.  The summation
$f+g:M\times\mathbb{R}\to\mathbb{R}$ defined as
$(f+g)(x,y)=f(x)+g(y)$ is a strictly convex function for the metric
$g_0+\de u^2$,  and it has a minimum point.  As in the proof of
Proposition 4.14 in \cite{CJM08}, we obtain the completeness of
$(N,\tilde F)$, $\tilde F$ defined in \eqref{RandersforE}. From
Remark~\ref{nonconj2}, the points $(x_0,0)$ and $(x_1,T)$ are
non-conjugate in $(N,\tilde F)$ and applying again
Theorem~\ref{geonfinito}, we complete the proof.
\end{proof}
\end{subsection}
\begin{subsection}{Zermelo's problem of navigation on Riemannian manifolds}\label{Zermelo}
The problem that we study in this section concerns the effects of a
mild wind in a Riemannian landscape $(M,g)$. This problem is known
as {\em Zermelo's navigation problem} (see \cite{Zermel31}) and it
was treated by C. Carath\'eodory in \cite{Carath67} when the
background is $\mathbb{R}^2$. Z. Shen has recently generalized  it
to arbitrary Riemannian backgrounds in any dimension (see
\cite{Shen03}). Following \cite{BaRoSh04},
we know that if the mild wind is represented by a vector field $W$
on $M$ such that $|W|<1$, for each $x\in M$, ($|\cdot|$ is the norm
associated to $g$) the trajectories that  minimize (or more
generally  make stationary) the travel time are the geodesics of the
metric
\begin{equation}\label{ZermeloMetric}
F(x,y)=\frac{\sqrt{g(W,y)^2+|y|^2\alpha(x)}}{\alpha(x)}-\frac{g(W,y)}{\alpha(x)},
\end{equation}
where $\alpha(x)=1-|W|^2$. Metrics as in \eqref{ZermeloMetric} are
of Randers type and  in \cite{BaRoSh04}   they are used to classify
Randers metrics with constant flag curvature,  while in
\cite{Robles07} a classification of their geodesics is obtained when
$W$ is an infinitesimal homothety. Moreover these metrics are very
similar to Fermat metrics in standard stationary spacetimes. The
only difference is that the one-form in the Randers metric has the
opposite sign and there is a constraint over $\beta$, that is,
$\beta(x)=1-|\delta|^2$. Thus, a Zermelo metric is a Fermat metric
with  $\beta=\alpha$ and $\delta=-W$. \bere\label{Zermelocomplete}
If  $\sup_{x\in M} |W|=\mu<1$   and the Riemannian metric $g$ is
complete,  then  the Zermelo metric is also forward and backward
complete. This is because  from equation (47) in \cite{CJM08} we
obtain
\[\sup_{x\in M}\|\omega_x\|\leq\sup_{x\in M}|W(x)|=\mu<1,\]
 where $\|\omega\|$ is the norm of the 1-form $\omega$ with respect to the metric
\[h(y,y)=\frac{1}{\alpha(x)}g(y,y)+\frac{1}{\alpha^2(x)}g(y,W)^2.\]
 As also $\frac{1}{\alpha}g$ is complete, applying Remark~\ref{completo}, we deduce the completeness of the Zermelo metric.
\ere From the above remark, the  result in Proposition \ref{improve}
can also be proved for Zermelo metrics.
\begin{proposition}\label{Zermelofinite}
Let $(M,g)$ be a complete Riemannian manifold, $W$  be a vector
field in $M$ such that  $\sup_{x\in M}|W(x)|=\mu<1$    and
$\alpha(x)=1-|W|^2$.  Assume that $(M,g)$ admits a $C^2$ convex
function $f\colon M\to\R$  having a minimum point and strictly
positive definite Hessian. If $\sup_{x\in M}\|\nabla W\|$ is small
enough, then there exists a finite number of Zermelo geodesics
joining  two  non-conjugate  points of $(M,F)$, being $F$ the
Randers metric defined  in \eqref{ZermeloMetric}.
\end{proposition}
\begin{proof}
The completeness of the Zermelo metric follows from
Remark~\ref{Zermelocomplete}. For the existence of the convex
function, it is enough to observe that
$$\frac{|\nabla(1-|W|^2)|}{1-|W|^2}=\frac{|\nabla|W|^2|}{1-|W|^2}< \frac{2}{1-\mu^2}|W|\|\nabla W\|<\frac{2}{1-\mu^2}\|\nabla W\|$$
and also to  apply  Proposition~\ref{improve}.
\end{proof}
\end{subsection}
\end{section}


\begin{thebibliography}{10}

\bibitem{bang}
Bangert, V.: {Riemannsche {M}annigfaltigkeiten mit nicht-konstanter konvexer {F}unktion}.
{Arch. Math. (Basel)}, \textbf{31},  163--170 (1978/79)

\bibitem{BaChSh00}
Bao, D., Chern, S.S., Shen, Z.: An Introduction to {R}iemann-{F}insler
  geometry.
\newblock Graduate Texts in Mathematics. Springer-Verlag, New York (2000)

\bibitem{BaRoSh04}
Bao, D., Robles, C., Shen, Z.: Zermelo navigation on {R}iemannian manifolds.
\newblock J. Differential Geom. \textbf{66}, 377--435 (2004)

\bibitem{BiJa08}
Biliotti, L., Javaloyes, M. A.:
$t$-periodic light rays in conformally stationary spacetimes via {F}insler geometry.
\href{http://arxiv.org/abs/0803.0488}{arXiv:0803.0488v2[math.DG]} (to appear in Houston J. Math.)  (2008)

\bibitem{BisO'N69}
Bishop, R.L., O'Neill, B.: Manifolds of negative curvature.
\newblock Trans. Amer. Math. Soc. \textbf{145}, 1--49 (1969)

\bibitem{CJM08}
Caponio, E., Javaloyes, M. A., Masiello, A.: On the energy
functional on Finsler manifolds and applications to stationary
spacetimes. \href{http://arxiv.org/abs/math/0702323}{arXiv:math/0702323v3 [math.DG]} (2008)

\bibitem{CJM09}
Caponio, E., Javaloyes, M. A., Masiello, A.:
{Morse theory of causal geodesics in a stationary spacetime via Morse
theory of geodesics of a Finsler metric}. {A}nn. {I}nst. {H}. {P}oincar\'e - {A}nal. {N}on {L}in\'eaire \textbf{27}, 857--876 (2010),
\href{http://arxiv.org/abs/0903.3519}{arXiv:0903.3519v3 [math.DG]}   

\bibitem{CJS}
Caponio, E., Javaloyes, M. A., S{\'a}nchez, M.:
On the interplay between Lorentzian causality and Finsler metrics
of Randers type,
\href{http://arxiv.org/abs/0903.3501}{arXiv:0903.3501v1 [math.DG]} (2009)


\bibitem{Carath67}
Carath{\'e}odory, C.: Calculus of Variations and Partial Differential Equations of the First Order.
\newblock Holden-Day Inc., San Francisco, Calif. (1967)

\bibitem{Carmo92}
do~Carmo, M.P.: Riemannian Geometry.
\newblock Mathematics: Theory \& Applications. Birkh\"auser, Boston, MA (1992)

\bibitem{CheGro72}
Cheeger, J., Gromoll, D.: On the structure of complete manifolds of nonnegative curvature.
\newblock Ann. of Math. (2) \textbf{96}, 413--443 (1972)

\bibitem{ggp}
Giamb{\`o}, R., Giannoni, F., Piccione, P.:
Genericity of nondegeneracy for light rays in stationary
spacetimes.
Comm. Math. Phys. \textbf{287}, 903--923 (2009)

\bibitem{gj}
Giamb{\`o}, R., Javaloyes, M. A.: Addendum to ``Genericity of nondegeneracy for light rays in stationary
spacetimes''. Comm. Math. Phys. \textbf{295}, 289--291 (2010)

\bibitem{GiMaPi99}
Giannoni, F., Masiello, A., Piccione, P.: Convexity and the finiteness of the number of geodesics. {A}pplications to the multiple-image effect.
\newblock Classical Quantum Gravity \textbf{16}, 731--748 (1999)

\bibitem{GiMaPi01}
Giannoni, F., Masiello, A., Piccione, P.: On the finiteness of light rays between a source and an observer on conformally stationary space-times.
\newblock Gen. Relativity Gravitation \textbf{33}, 491--514 (2001)

\bibitem{GMPJMP}
Giannoni, F., Masiello, A., Piccione, P.: The Fermat principle in
General Relativity and applications.  J. Math. Phys. {\bf 43},
563-596 (2002)

\bibitem{GreShi81}
Greene, R. E., Shiohama, K.:
Convex functions on complete noncompact manifolds: differentiable structure.
{Ann. Sci. \'Ecole Norm. Sup. (4)}, \textbf{14}, 357--367 (1981)

\bibitem{GreShi81a}
Greene, R. E., Shiohama, K.:
Convex functions on complete noncompact manifolds: topological structure.
Invent. Math., \textbf{63}, 129--157 (1981)

\bibitem{Hassan}
Hassan, B.T.: The cut locus of a Finsler manifold.
Atti Accad. Naz. Lincei Rend. Cl. Sci. Fis. Mat. Natur. \textbf{54}, 739--744 (1974)

\bibitem{javpic}
Javaloyes, M. A., Piccione, P.:
On the singularities of the semi-{R}iemannian exponential map. {B}ifurcation
of geodesics and light rays.
In ``Variations on a century of relativity: theory and applications'',
Lect. Notes Semin. Interdiscip. Mat., V.
S.I.M. Dep. Mat. Univ. Basilicata, Potenza, 115--123 (2006)

\bibitem{JS08}
Javaloyes M. A., S{\'a}nchez, M.:
A note on the existence of standard splittings for conformally stationary spacetimes.
Classical Quantum Gravity, \textbf{25}  pp.~168001, 7 (2008)

\bibitem{PE}
Perlick V.: Gravitational lensing from a spacetime perspective.
Living Reviews Relativity {\bf 7}, (2004).

\bibitem{Rander41}
Randers, G.: On an asymmetrical metric in the fourspace of {G}eneral {R}elativity.
\newblock Phys. Rev. \textbf{59}, 195--199 (1941)

\bibitem{Robles07}
Robles, C.:
Geodesics in {R}anders spaces of constant curvature.
Trans. Amer. Math. Soc., \textbf{359}, 1633--1651 (2007)


\bibitem{Shen03}
Shen, Z.: Finsler metrics with {$\mathbf{K}=0$} and {$\mathbf{S}=0$}.
\newblock Canad. J. Math. \textbf{55}, 112--132 (2003)

\bibitem{Udriste}
{Udri{\c{s}}te, C.}: {Convex functions and optimization methods on
  {R}iemannian manifolds},
  \newblock vol.~297 of Mathematics and its Applications, Kluwer
  Academic Publishers Group, Dordrecht, 1994


\bibitem{Zermel31}
Zermelo, E.: \"{U}ber das {N}avigationsproblem bei ruhender oder
  ver\"anderlicher {W}indverteilung.
\newblock Z. Angew. Math. Mech. \textbf{11}, 114--124 (1931)

\end{thebibliography}
\end{document}